\documentclass[11pt,a4paper,leqno]{article}
\usepackage[margin=1.1in]{geometry}
\usepackage{amsmath,amsthm,amssymb,enumerate, bbm ,graphicx,color,caption,upgreek, float, wasysym, subcaption,booktabs,longtable, appendix,subcaption,graphics, pdfpages,rotating,mathtools,subcaption,array,supertabular}
\usepackage[justification=justified,singlelinecheck=false]{caption}

\definecolor{donkergroen}{RGB}{46,148,0}
\usepackage[hidelinks]{hyperref} 
\definecolor{donkerrood}{RGB}{204,0,0}

\definecolor{blauw}{RGB}{61,158,255}
\definecolor{donkerblauw}{RGB}{0,0,255}
\definecolor{donkergroen}{RGB}{46,148,0}
\definecolor{donkerrood}{RGB}{204,0,0}

\usepackage{abstract}

\makeatletter
\newif\if@borderstar
\def\bordermatrix{\@ifnextchar*{%
\@borderstartrue\@bordermatrix@i}{\@borderstarfalse\@bordermatrix@i*}%
}
\def\@bordermatrix@i*{\@ifnextchar[{\@bordermatrix@ii}{\@bordermatrix@ii[()]}}
\def\@bordermatrix@ii[#1]#2{%
\begingroup
\m@th\@tempdima8.75\p@\setbox\z@\vbox{%
\def\cr{\crcr\noalign{\kern 2\p@\global\let\cr\endline }}%
\ialign {$##$\hfil\kern 2\p@\kern\@tempdima & \thinspace %
\hfil $##$\hfil && \quad\hfil $##$\hfil\crcr\omit\strut %
\hfil\crcr\noalign{\kern -\baselineskip}#2\crcr\omit %
\strut\cr}}%
\setbox\tw@\vbox{\unvcopy\z@\global\setbox\@ne\lastbox}%
\setbox\tw@\hbox{\unhbox\@ne\unskip\global\setbox\@ne\lastbox}%
\setbox\tw@\hbox{%
$\kern\wd\@ne\kern -\@tempdima\left\@firstoftwo#1%
\if@borderstar\kern2pt\else\kern -\wd\@ne\fi%
\global\setbox\@ne\vbox{\box\@ne\if@borderstar\else\kern 2\p@\fi}%
\vcenter{\if@borderstar\else\kern -\ht\@ne\fi%
\unvbox\z@\kern-\if@borderstar2\fi\baselineskip}%
\if@borderstar\kern-2\@tempdima\kern2\p@\else\,\fi\right\@secondoftwo#1 $%
}\null \;\vbox{\kern\ht\@ne\box\tw@}%
\endgroup
}
\makeatother

\makeatletter 
\newcommand\mynobreakpar{\par\nobreak\@afterheading} 
\makeatother

\newcommand{\N}{\mathbb{N}}
\newcommand{\Z}{\mathbb{Z}}
\newcommand{\C}{\mathbb{C}}
\newcommand{\R}{\mathbb{R}}

\newcommand{\CC}{\mathcal{C}}

\usepackage[english]{babel}

\newtheorem{theorem}{Theorem}[section]
\newtheorem{lemma}[theorem]{Lemma}
\newtheorem{claim}[theorem]{Claim}
\newtheorem{proposition}[theorem]{Proposition}

\theoremstyle{definition}
\newtheorem{defn}{Definition}[section]

\newtheorem*{examp*}{Example}

\theoremstyle{plain}

\newfloat{Algorithm}{!hbt}{alg}

\newcounter{thm}[section]

\title{{\large \textbf{SEMIDEFINITE BOUNDS FOR MIXED BINARY/TERNARY CODES}}}
\author{{\normalsize Bart Litjens}\footnote{Korteweg-De Vries Institute for Mathematics, University of Amsterdam, Amsterdam, The Netherlands. The research leading to these results has received funding from the European Research Council under the European Union's Seventh Framework Programme (FP7/2007-2013) / ERC grant agreement n$^{\circ}$ 339109.}}
\date{\vspace{-5ex}}
\begin{document}
\maketitle

\noindent {\small \textbf{Abstract.} For nonnegative integers $n_2, n_3$ and $d$, let $N(n_2,n_3,d)$ denote the maximum cardinality of a code of length $n_2+n_3$, with $n_2$ binary coordinates and $n_3$ ternary coordinates (in this order) and with minimum distance at least $d$. For a nonnegative integer $k$, let $\CC_k$ denote the collection of codes of cardinality at most $k$. For $D \in \CC_k$, define $S(D) \coloneqq \{C \in \CC_k \mid D \subseteq C, |D| +2|C\setminus D| \leq k\}$. Then $N(n_2,n_3,d)$ is upper bounded by the maximum value of $\sum_{v \in [2]^{n_2}[3]^{n_3}}x(\{v\})$, where $x$ is a function $\CC_k \rightarrow \R$  such that $x(\emptyset) = 1$ and $x(C) = 0$ if $C$ has minimum distance less than $d$, and such that the $S(D)\times S(D)$ matrix $(x(C\cup C'))_{C,C' \in S(D)}$ is positive semidefinite for each $D \in \CC_k$. By exploiting symmetry, the semidefinite programming problem for the case $k=3$ is reduced using representation theory. It yields $135$ new upper bounds that are provided in tables.}\\

\noindent \textbf{Key words:} code, mixed binary/ternary code, upper bounds, semidefinite programming
\noindent \textbf{MSC 2010:} 94B65, 05E10, 90C22, 20C30

\section{Introduction}\label{section:Introduction}

Let $\Z_{+}$ be the set of nonnegative integers, and let $[n] = \{1,...,n\}$, for any $n \in \Z_{+}$. Let $n_2,n_3 \in \Z_{+}$ be fixed. Then a \textit{mixed binary/ternary code} is a subset of $[2]^{n_2}[3]^{n_3}$. Mixed codes are of interest because of their application to football pools, see for instance \cite{hamalainen95}. Whenever $[n]$ consists of the \textit{letters} of an \textit{alphabet} of a code, we take the letters mod $n$. Since all codes considered in this paper are mixed, i.e., both $n_2 > 0$ and $n_3 > 0$, we will speak of \textit{codes} from now on. An element of a code is called a \textit{codeword} or \textit{word}.\\
\indent Given two words $v,w \in [2]^{n_2}[3]^{n_3}$, the \textit{Hamming distance} $d_{H}(v,w)$ between $v$ and $w$ is the number of positions $i \in [n_2+n_3]$ for which $v_i \neq w_i$. The Hamming distance between a word $v$ and the all-zero word is called the \textit{weight} of $v$, denoted $w(v)$. For a code $C$, the \textit{minimum distance} of $C$ is equal to the minimum of $d_H(v,w)$, where we range over distinct $v,w \in C$. Note that with this definition, the empty code and codes of size one do not have a minimum distance. The maximum cardinality of a code with minimum distance at least $d$ is denoted by $N(n_2,n_3,d)$. We will define a hierarchy of upper bounds on $N(n_2,n_3,d)$ that sharpens the linear programming bound defined in \cite{brouwer98}.\\
\indent For $k \in \Z_{+}$, let $\CC_k$ denote the collection of codes of cardinality at most $k$. For $D \in \CC_k$, define $S(D) \coloneqq \{C \in \CC_k \mid D \subseteq C, |D| +2|C\setminus D| \leq k\}$. Note that $|C \cup C'| \leq k$, for $C, C' \in S(D)$. For each function $x: \CC_k \rightarrow \R$, and for each $D \in \CC_k$, define the $S(D) \times S(D)$ matrix $M_D(x) = (x(C \cup C'))_{C,C' \in S(D)}$. Then we define 
\begin{align}\label{align:upperbound}
&N_{k}(n_2,n_3,d) \coloneqq \max_{x} \hspace{-1ex} \sum_{v \in [2]^{n_2}[3]^{n_3}} \hspace{-1ex}x(\{v\}), \text{where }x: \CC_k \rightarrow \R \text{ satisfies}\\
& \text{(i)} \hspace{1mm} x(\emptyset) = 1, \nonumber \\
& \text{(ii)} \hspace{1mm} x(C) = 0 \text{ if the minimum distance of }C \text{ is less than }d,\nonumber \\
& \text{(iii)} \hspace{1mm} M_D(x) \text{ is positive semidefinite for each }D \in \CC_k. \nonumber
\end{align}
\indent Observe that for a code $D$ of size $k$, positive semidefiniteness of $M_D(x)$ is equivalent to nonnegativity of $x(D)$. Hence, in (\ref{align:upperbound}), we could as well assume that $x: \CC_k \rightarrow \R_+$.
\begin{proposition}\label{proposition:upperbound}
For $n_2,n_3,d,k \in \Z_{+}$, it holds that $N(n_2,n_3,d) \leq N_k(n_2,n_3,d)$. 
\end{proposition}
\begin{proof}
Let $D \subseteq [2]^{n_2}[3]^{n_3}$ be of minimum distance at least $d$, such that $|D| =~N(n_2,n_3,d)$. Define $x: \CC_k \rightarrow \R$ by $x(C) = 1$ if $C \subseteq D$ and $x(C) = 0$ otherwise. This function clearly satisfies conditions (i) and (ii) of (\ref{align:upperbound}). Since $(M_{D}(x))_{C,C'} =~x(C)x(C')$ for all $C,C' \in \CC_k$, condition (iii) is also satisfied. Now $\sum_{v \in [2]^{n_2}[3]^{n_3}} \hspace{-0ex}x(\{v\}) = |D| = N(n_2,n_3,d)$, and the proposition follows.
\end{proof}
In this paper, we consider $k=3$. The optimization problem (1) for triples of codewords is very large. However, the problem is highly symmetric and therefore representation theory of the symmetric group can be applied in order to reduce the dimensions to size bounded by a polynomial in $n_2$ and $n_3$. This enables us to solve (1) by semidefinite programming for many choices of triples $(n_2,n_3,d) \in \N^{3}$. We will now describe the ideas of the reduction. The precise details may be found in Section \ref{section:reduction}.\\
\indent Let $G$ be the isometry group of $[2]^{n_2}[3]^{n_3}$. That is, $G$ is the group of Hamming distance-preserving bijections from $[2]^{n_2}[3]^{n_3}$ to itself. Then $G = H_2 \times H_3$, where $H_2$ is the wreath product $S_2^{n_2} \rtimes S_{n_2}$ and $H_3$ is the wreath product $S_3^{n_3} \rtimes S_{n_3}$. Here, $S_{m}$ denotes the symmetric group on $m$ letters. For $i=2,3$, an element $h \in H_i$ permutes the $n_i$ coordinates and permutes the letters in $[i]$ in every of the $n_i$ positions. The group $G$ acts on $\CC_k$ and hence on functions $x: \CC_k \rightarrow \R$, via $x^{\pi}(C) \coloneqq x(\pi^{-1}(C))$, for $\pi \in G$ and $C \in \CC_k$. By definition of $G$, minimum distances of codes are preserved under this action. Let $x: \CC_k \rightarrow \R$ be a function satisfying the conditions and maximizing the objective function of (\ref{align:upperbound}). For $\pi \in G$, the function $x^{\pi}$ again satisfies conditions (i) and (ii) of (\ref{align:upperbound}). Condition (iii) is met as well, as the matrix $M_D(x^{\pi})$ is obtained from $M_D(x)$ by simultaneously permuting rows and columns. Since $\pi$ is a bijection of $[2]^{n_2}[3]^{n_3}$, the objective function does not change when replacing $x$ by $x^{\pi}$. Averaging over the group $G$ yields a $G$-invariant function $y$, for which the matrices $M_D(y)$ are positive semidefinite by convexity of the set of positive semidefinite matrices. This shows that the optimal function $x$ can be taken to be $G$-invariant.\\
\indent Let $\Omega$ be the set of orbits of $\CC_k$ under the action of $G$. Since a $G$-invariant function $y$ is constant on orbits, for each $D \in \CC_k$ the matrix $M_D(y)$ can be written in terms of variables $y(w)$, with $w \in \Omega$. Let $G_{D}$ be the subgroup of $G$ that leaves $D$ invariant. Then $M_D(y)$ is invariant under the induced action of $G_D$ on its rows and columns. Therefore, it admits a block-diagonalization $M_D(y) \mapsto U^{T}M_D(y)U$, where $U$ is a matrix independent of $y$ (see equation (\ref{equation:phireal})). The matrix $M_D(y)$ is positive semidefinite if and only if each of the blocks is. This accounts for a large reduction as the blocks have far less entries than the original matrix, and the same block occurs repeatedly.\\
\indent For $D \in \CC_k$ and $\pi \in G$, the matrix $M_D(y)$ differs from $M_{\pi(D)}(y)$ by a permutation matrix. Hence, positive semidefiniteness of $M_D(y)$ needs only be checked for one element $D$ out of each $G$-orbit of $\CC_k$. Throwing away equivalent blocks, we are left with blocks whose entries are linear functions in the variables $y(w)$. The number of variables is bounded by a polynomial in $n_2$ and $n_3$, see Section \ref{subsection:sizeonecoef}.\\
\indent The blocks as well as some further reductions of the optimization problem will be described in Section \ref{section:reduction}. The entries of the matrices are computed in Section \ref{section:coefficients}. Table \ref{table:bounds} at the end of the article shows the improvements that were found using the multiple precision versions of the semidefinite programming algorithm SDPA, with thanks to SURFsara (\url{www.surfsara.nl}) for the support in using the LISA Compute Cluster.\\
\indent Several previously best known upper bounds were obtained via linear programming and extra constraints in \cite{brouwer98} by Brouwer,  H\"am\"al\"ainen, \"Osterg\aa rd and Sloane. For $d=3$ and $d=4$, improvements were found by {\"O}sterg{\aa}rd using backtrack search in \cite{ostergard00} and \cite{ostergard99} respectively. The tables in \cite{brouwer16}, maintained by Andries Brouwer, contain all known bounds on the size of binary/ternary error-correcting codes.\\

\subsection{Comparison with earlier bounds}

The above described method is an adaption of the one in \cite{litjens17} and builds upon the work of Gijswijt, Mittelmann, Schrijver and Tanaka in \cite{gijswijt12}, \cite{gijswijt06}, \cite{schrijver05}. Proposition \ref{proposition:upperbound} generalizes Proposition $1$ of \cite{litjens17} for the binary and ternary case. In fact, for fixed $t \in \Z_{+}$ and distinct $p_1,...,p_t \in \N$, the statement in Proposition \ref{proposition:upperbound} can be generalized to the case of mixed codes of length $n_1 + ... + n_t$, with $n_i$ coordinates chosen from an alphabet with $p_i$ letters, for $i = 1,...,t$.\\
\indent The method described in the previous section (with $k=3$) fits into the second level of the Lasserre hierarchy for stable sets. It can be proved that for $k=2$, Proposition \ref{proposition:upperbound} reduces to the pure linear programming bound described in Section $2$ of \cite{brouwer98}.\\
\indent Theoretically, our method could be extended to $k \geq 4$. However, the number of variables involved in the semidefinite program grows rapidly when going from $k=3$ to $k=4$. In practice, for $k=4$ only one case could be made tractable. Furthermore, the instances in the tables in \cite{brouwer16} where the value $N(n_2,n_3,d)$ is yet unsettled, typically involve codes for which the length $n_2+n_3$ is large compared to the distance $d$. This amounts to many and large constraint matrices. 

\section{Preliminaries on representation theory}

In this section some background information on group actions and representation theory of finite groups is given. It mostly concerns representation theory of the symmetric group. Proofs and details of the statements given are omitted. For these we refer the reader to chapters $1$ and $2$ of Sagan's book \cite{sagan01}. Furthermore, this section is intended to set up the notation that is used throughout the article.\\
\indent Let $G$ be a finite group and $X$ a set. Let $S_X$ denote the group of bijections from $X$ to itself. A \textit{group action} from $G$ on $X$ is a group homomorphism $G \rightarrow S_X$. If $G$ acts on $X$, we denote $g \cdot x$ for the image of $x$ under the bijection associated to $g$, where $x \in X$ and $g \in G$. If $X$ is \textit{linear}, elements of $S_X$ are also assumed to be linear. This applies for example to the following situation. For a field $K$ and a set $X$, let $K^X$ denote the linear space of maps from $X$ to $K$. If $G$ acts on $X$, then $G$ acts on $K^X$ by $(g \cdot f)(x) \coloneqq f(g^{-1} \cdot x)$, for all $g \in G, f \in K^X$ and $x \in X$. Lastly, by $X^G$ we denote the set of elements of $X$ that are left invariant by all of $G$.\\
\indent The following review of the representation theory of finite groups is not as general as possible, but rather concrete, which suits our purposes. Let $m \in \Z_+$ and let $V = \C^{m}$ be acted upon by a finite group $G$. Then $V$ is called a \textit{G-module}. If $W$ is another $G$-module, a \textit{G-homomorphism} from $V$ to $W$ is a linear map $\phi: V \rightarrow W$ such that $g \cdot \phi(v) = \phi(g \cdot v)$, for all $g \in G$ and $v \in V$. The module $V$ is called \textit{irreducible} if it has no nontrivial $G$-invariant submodules.\\
\indent Assume now that $G$ acts \textit{unitarily} on $V$. This means that for every $g \in G$ there is a unitary matrix $U$ such that $g \cdot v = Uv$ for all $v \in V$. Then the standard inner product $\langle v,w \rangle = v^*w$ on $V$, where $^*$ denotes the complex conjugate, is a $G$-invariant inner product, i.e., $\langle g\cdot v, g\cdot w\rangle = \langle v,w\rangle$ for all $g \in G$ and $v,w \in V$. If $U \subset V$ is a submodule, then so is $U^{\perp} \coloneqq \{v \in V \hspace{1mm} | \hspace{1mm} \langle v,u\rangle = 0 \hspace{1mm} \forall u \in U\}$. This shows that $V$ admits a decomposition into pairwise orthogonal irreducible submodules (Maschke's theorem). Grouping mutually isomorphic submodules, we write $V = V_1 \oplus ... \oplus V_k$ as a direct sum of \textit{isotypic components}. For each $i \leq k$, there is an $m_i \in \N$, called the \textit{multiplicity} of $V_{i,1}$ in $V$, such that $V_i = V_{i,1} \oplus ... \oplus V_{i,m_i}$. We have that $V_{i,j}$ and $V_{i',j'}$ are isomorphic irreducible $G$-modules if and only if $i = i'$.\\
\indent With notation as above, Schur's lemma implies that the space of $G$-endomorphisms of $V$ is linearly isomorphic to a direct sum of matrix algebras with sizes given by the multiplicities:
\[
(\C^{m \times m})^{G} \cong \bigoplus_{i=1}^{k}\C^{m_i \times m_i}.
\]
We describe an explicit isomorphism. For every $i \leq k$ and $j \leq m_i$, choose a nonzero vector $u_{i,j} \in V_{i,j}$ such that for every $i \leq k$ and $j,j' \leq m_i$ there exists a $G$-isomorphism from $V_{i,j}$ to $V_{i,j'}$ that maps $u_{i,j}$ to $u_{i,j'}$. Consider the matrix $U_i = [u_{i,1},...,u_{i,m_i}]$ for $i \leq k$ whose columns are given by the vectors $u_{i,j}$. 
\begin{defn}
In the situation as described above, any set of matrices $\{U_1,...,U_k\}$ is called a \textit{representative set} for the action of $G$ on $V$.
\end{defn}
If $\{U_1,...,U_k\}$ is a representative set, then the function
\begin{equation}\label{equation:phi}
\Phi: (\C^{m \times m})^{G} \rightarrow \bigoplus_{i=1}^{k}\C^{m_{i} \times m_{i}}, \hspace{2mm} A \mapsto \bigoplus_{i=1}^{k}U_{i}^{*}AU_{i},
\end{equation}
is a linear isomorphism (see Theorem $3$ of \cite{gijswijt14} for a proof). Recall that a complex-valued matrix is \textit{positive semidefinite} if it is a Hermitian matrix whose eigenvalues are all nonnegative. An important property of $\Phi$ is that both $\Phi$ and its inverse preserve positive semidefiniteness.\\
\indent In this article, the previous is applied to the case where a finite group $G$ acts \textit{real-orthogonally} on a vector space $V = \R^{m}$. This means that for every $g \in G$ there is a real orthogonal matrix $U$ such that $g \cdot v = Uv$ for every $v \in V$. We will describe a representative set $\{U_1,...,U_k\}$ for the action of $G$ on $V$ consisting of real matrices. In that situation, $V$ can be decomposed as
\[
V = \bigoplus_{i=1}^{k}\bigoplus_{j=1}^{m_i}\R G\cdot u_{i,j},
\]
where $\R G$ is the \textit{group algebra} of $G$. The map $\Phi$ in (\ref{equation:phi}) becomes
\begin{equation}\label{equation:phireal}
\Phi: (\R^{m \times m})^{G} \rightarrow \bigoplus_{i=1}^{k}\R^{m_{i} \times m_{i}}, \hspace{2mm} A \mapsto \bigoplus_{i=1}^{k}U_{i}^{T}AU_{i},
\end{equation}
where $^{T}$ denotes taking the transpose. Then $A$ is positive semidefinite if and only if each of the blocks $U_i^{T}AU_i$ is. For reasons that become apparent later, we view the columns $u_{i,j}$ of the matrices in the representative set as elements of the \textit{dual space} $V^{*}$ via the $G$-invariant inner product.

\subsection{A representative set for the action of $S_n$ on $V^{\otimes n}$}

For $n \in \N$, consider the action of the symmetric group $S_n$ on a finite dimensional real vector space $V^{\otimes n}$ by permuting the indices. We will describe a representative set for this action in terms of \textit{semistandard Young tableaux}.\\
\indent A \textit{partition} $\lambda$ of $n$ is a sequence of natural numbers $\lambda_1 \geq ... \geq \lambda_t > 0$ such that $n = \lambda_1 + ... + \lambda_t$. The number $t$ is called the \textit{height} of $\lambda$. If $\lambda$ partitions $n$, we write $\lambda \vdash n$ to indicate this. With respect to a partition $\lambda \vdash n$ of height $t$, we define the \textit{Ferrers diagram} $Y(\lambda)$ as
\[
Y(\lambda) \coloneqq \{(i,j) \in \Z_+^2 \hspace{1mm} | \hspace{1mm} 1 \leq j \leq t, 1 \leq i \leq \lambda_j\}.
\]
Fixing a $j \leq t$, the elements $(i,j)$ in $Y(\lambda)$ where $i$ varies, form the \textit{j-th row} of $Y(\lambda)$. Likewise, when an $i \leq \lambda_1$ is fixed and the $j$ vary, the elements $(i,j)$ in $Y(\lambda)$ form the \textit{i-th column}. With respect to $\lambda$, we define two subgroups of $S_{Y(\lambda)}$. The group $R_{\lambda}$ is the subgroup of $S_{Y(\lambda)}$ consisting of permutations $\pi$ such that $\pi(Z) = Z$ for each row $Z$ of $Y(\lambda)$. It is called the \textit{row stabilizer}. The group $C_{\lambda}$ contains all permutations $\pi$ such that $\pi(Z) = Z$ for all columns $Z$ of $Y(\lambda)$ and is called the \textit{column stabilizer}.\\
\indent Let $\lambda \vdash n$. For $m \in \Z_+$, a \textit{Young tableau} with entries in $[m]$  is a function $\tau: Y(\lambda) \rightarrow [m]$. Two Young tableaux $\tau$ and $\tau'$ are called \textit{row equivalent}, written $\tau \sim \tau'$, if there exists a $\pi \in R_{\lambda}$ such that $\tau' =\tau\pi$. A Young tableau is \textit{semistandard} if in each row the entries are nondecreasing and if in each column the entries are increasing. By $T_{\lambda, m}$ we denote the set of semistandard Young tableaux with entries in $[m]$. Note that $T_{\lambda,m}$ is nonempty if and only if $m$ is larger than or equal to the height of $\lambda$.\\
\indent  Let $(B(1),...,B(m))$ be an \textit{ordered} basis of the dual space $V^*$. For a Young tableau $\tau: Y(\lambda) \rightarrow [m]$, we define
\[
u_{\tau,B} \coloneqq \sum_{\tau' \sim \tau}\sum_{c \in C_{\lambda}}\text{sgn}(c)\bigotimes_{y \in Y(\lambda)} B(\tau' c(y)).
\]
Here, we order $Y(\lambda)$ by concatenating the rows, starting from the first row. The matrix set
\[
\{ \hspace{2mm} [u_{\tau,B} \hspace{1mm} | \hspace{1mm} \tau \in T_{\lambda,m}] \hspace{2mm} | \hspace{2mm} \lambda \vdash n\}
\]
is a representative set for the action of $S_n$ on $V^{\otimes n}$.

\section{Reduction of the optimization problem}\label{section:reduction}

In this section we describe the reduction of the optimization problem (\ref{align:upperbound}), using the notation set up in the previous sections. This is done by finding representative sets for the action of $G_D$ on $\R^{S(D)}$ for one code $D$ out of each orbit $w$ in $\Omega$. Fix $n_2,n_3,d \in \Z_+$ and set $k=3$. If a code $D \in \CC_k$ has size $2$ or $3$, then $S(D) = \{D\}$ and $M_D(y) = (y(D))$. Condition (iii) of (\ref{align:upperbound}) then amounts to nonnegativity of the variable $y(D)$. Subsequently, we need only to deal with codes $D$ with $|D| = 0$ or $|D| = 1$. 

\subsection{A code of size one}\label{subsection:sizeone}

Since the isometry group $G$ acts transitively on $[2]^{n_{2}}[3]^{n_{3}}$, we may assume that a code $D$ of size one consists of the all-zero word. The rows and columns of $M_{D}(y)$ are parametrized by pairs of words that contain the all-zero word. The stabilizer subgroup $G_D$ of $D$ in $G$ equals $S_{n_{2}} \times (S_{2}^{n_{3}} \rtimes S_{n_{3}})$. To obtain a representative set for the action of $G_{D}$ on $\R^{S(D)}$, we first describe a representative set for the action of $G_{D}$ on $\R^{[2]^{n_2}[3]^{n_3}}$ and then restrict to words of weight zero or at least $d$.\\
\indent In order to obtain a representative set, consider independently the action of the trivial group on $\R^{[2]}$ and the action of $S_2$ on $\R^{[3]}$, permuting the nonzero letters. Let $e_{j}$ be the $j$-th unit vector of $\R^{[2]}$, with $j = 1,2$ and let $f_{l}$ be the $l$-th unit vector of $\R^{[3]}$, with $l = 1,2,3$. Define the following matrices
\begin{align}\label{align:atjes}
&A_1 \coloneqq [e_1, e_2], \hspace{1mm} A_2 \coloneqq [f_1,f_2+f_3] \text{ and} \hspace{1mm} A_3 \coloneqq [f_2-f_3], 
\end{align}
where we view the vectors as columns vectors. Then $\{A_1\}$ and $\{A_2,A_3\}$ form representative sets for the actions just described\footnote{The vectors $e_1, e_2$ and $f_1, f_2+f_3$ span different copies of the trivial representation inside $\R^{[2]}$ and $\R^{[3]}$ respectively. The vector $f_2-f_3$ spans a copy of the sign representation of $S_2$ inside $\R^{[3]}$.}. \\
\indent Set $m_1= m_2=2$ and $m_3=1$ and let $\mathbf{N_{1}}$ denote the set of triples $(n_2,l_2,l_3) \in \Z_+^{3}$ such that $l_2 + l_3 = n_3$. For $\mathbf{n} = (n_2,l_2,l_3) \in \mathbf{N_{1}}$, by $\pmb{\lambda} \vdash \mathbf{n}$ we indicate that $\pmb{\lambda} = (\lambda_1,\lambda_2,\lambda_3)$ with $\lambda_1 \vdash n_2, \lambda_2 \vdash l_2$ and $\lambda_3 \vdash l_3$. Let $\pmb{\lambda} \vdash \mathbf{n}$, then we define
\[
W_{\pmb{\lambda}} \coloneqq \prod_{i=1}^{3}T_{\lambda_{i},m_{i}}.
\]
For $\pmb{\tau} = (\tau_{1},\tau_{2},\tau_{3}) \in W_{\pmb{\lambda}}$ we define
\begin{equation}\label{equation:tensora}
u_{\pmb{\tau}} \coloneqq \bigotimes_{i=1}^{3} u_{\tau_{i},A_{i}}.
\end{equation}
Then Proposition $2$ of \cite{litjens17} implies the following.

\begin{proposition}
The matrix set
\[
\{ \hspace{2mm} [u_{\pmb{\tau}} \hspace{1mm} | \hspace{1mm} \pmb{\tau} \in W_{\pmb{\lambda}} ] \hspace{2mm} | \hspace{2mm} \mathbf{n} \in \mathbf{N_{1}}, \pmb{\lambda} \vdash \mathbf{n}\}
\]
is a representative set for the action of $G_{D}$ on $\R^{[2]^{n_2}[3]^{n_3}}$.
\end{proposition}

Next we reduce to words of weight zero or at least $d$. For a word $v \in [2]^{n_2}[3]^{n_3}$, write $v = v_2v_3$ with $v_2 \in [2]^{n_2}$ and $v_3 \in [3]^{n_3}$. Then we define the vector 
\[
w_v \coloneqq (w(v_2),w(v_3)),
\]
in $\Z_+^2$, with $w(v_i)$ the weight of $v_i$. Given $w = (w_2,w_3) \in \Z_+^2$, let $V_w$ denote the linear subspace of $\R^{[2]^{n_2}[3]^{n_3}}$ spanned by unit vectors $e_v$, with $v$ a word for which $w_v = w$. For any $u_{\pmb{\tau}}$ with $\pmb{\tau} = (\tau_{1},\tau_{2},\tau_{3})$ as in (\ref{equation:tensora}) the irreducible representation $\R G_D \cdot u_{\pmb{\tau}}$ is contained in $V_w$, where $w = (w_2,w_3)$ with
\[
w_2 = n_2 - |\tau_{1}^{-1}(1)| \hspace{1mm} \text{ and } \hspace{1mm} w_3 = n_3 - |\tau_{2}^{-1}(1)|.
\]
Indeed, every permutation of $G_{D}$ leaves the weight of a word invariant. We now define
\[
W_{\pmb{\lambda}}' \coloneqq \{\pmb{\tau} \in W_{\pmb{\lambda}} \hspace{1mm} | \hspace{1mm} n_2+n_3 - |\tau_{1}^{-1}(1)| - |\tau_{2}^{-1}(1)| \in \{0,d,d+1,...,n_2+n_3\}\}.
\]
Then a representative set for the action of $G_{D}$ on $\R^{S(D)}$ is given by the matrix set
\begin{equation}\label{equation:reprsetg}
\{ \hspace{2mm} [u_{\pmb{\tau}} \hspace{1mm} | \hspace{1mm} \pmb{\tau} \in W_{\pmb{\lambda}}' ] \hspace{2mm} | \hspace{2mm} \mathbf{n} \in \mathbf{N_{1}}, \pmb{\lambda} \vdash \mathbf{n}\}.
\end{equation}

\subsection{$\boldsymbol{D = \emptyset}$}\label{subsection:sizezero}

Let $D = \emptyset$. Then $S(D)$ is the collection of singletons together with the empty set and $G_{D} = G$. To obtain a representative set for the action of $G_{D}$ on $\R^{S(D)}$, we first consider the action of $G_{D}$ on $\R^{[2]^{n_2}[3]^{n_3}}$ and later add the empty code.\\
\indent For $i=2,3$, let $S_{i}$ act on $\R^{[i]}$ by permuting the letters. Representative sets are given by\footnote{The vector $e_1+e_2$ spans a copy of the trivial representation of $S_2$ in $\R^{[2]}$ and the vector $e_1-e_2$ accounts for the sign representation. The space $\R^{[3]}$ decomposes as a $S_3$-module into the standard representation, spanned by for example $f_1-f_2$ and $f_2-f_3$, and the trivial representation, spanned by $f_1+f_2+f_3$.} $\{B_1,B_2\}$ for $i=2$ and $\{B_3,B_4\}$ for $i=3$, where
\begin{align}\label{align:b1234}
&B_1 \coloneqq [e_1+e_2], \hspace{1mm} B_2 \coloneqq [e_1-e_2], \hspace{1mm} B_3 \coloneqq [f_1+f_2+f_3] \text{ and} \hspace{1mm} B_4 \coloneqq [f_1-f_2].
\end{align}
\indent Set $m_1=m_2=m_3=m_4=1$ and let $\mathbf{N_{0}}$ denote the set of quadruples $(l_1,l_2,l_3,l_4) \in \Z_+^4$ such that $l_1 + l_2 = n_2$ and $l_3+l_4=n_3$. For $\mathbf{n} = (l_1,l_2,l_3,l_4) \in \mathbf{N_{0}}$, by $\pmb{\lambda} \vdash \mathbf{n}$ we indicate that $\pmb{\lambda} = (\lambda_1,\lambda_2,\lambda_3,\lambda_4)$ with $\lambda_i \vdash l_i$ for $1 \leq i \leq 4$. Let $\pmb{\lambda} \vdash \mathbf{n}$, then we define
\[
Z_{\pmb{\lambda}} \coloneqq \prod_{i=1}^{4} T_{\lambda_i,m_i}.
\]
For $\pmb{\tau} = (\tau_1,\tau_2,\tau_3,\tau_4) \in Z_{\pmb{\lambda}}$ we define
\[
v_{\pmb{\tau}} \coloneqq \bigotimes_{i=1}^{4}u_{\tau_{i},B_i}.
\]
Using Proposition $2$ of \cite{litjens17} again yields the following representative set.

\begin{proposition}\label{proposition:reprsetleeg}
The matrix set
\begin{equation}\label{equation:reprsetleeg}
\{ \hspace{2mm} [v_{\pmb{\tau}} \hspace{1mm} | \hspace{1mm} \pmb{\tau} \in Z_{\pmb{\lambda}} ] \hspace{2mm} | \hspace{2mm} \mathbf{n} \in \mathbf{N_{0}}, \pmb{\lambda} \vdash \mathbf{n}\}
\end{equation}
is a representative set for the action of $G_{D}$ on $\R^{[2]^{n_2}[3]^{n_3}}$.
\end{proposition}
Next we have to add the empty code $D$. Since $G_{D}$ acts trivially on $D$, the vector $e_{\emptyset}$ should be added to the $G_D$-isotypic component that consists of the $G_D$-invariants. This is the matrix indexed by the partition $\pmb{\lambda} = ((n_2),(),(n_3),())$ of $\mathbf{n} = (n_2,0,n_3,0)$. Here, $()$ denotes the partition of zero and $(n_i)$ the partition of $n_i$ of height one, for $i=2,3$. 

\section{Computation of the coefficients}\label{section:coefficients}

In the previous section representative sets for the action of $G_D$ on $\R^{S(D)}$ were found for the case that $D$ is the empty code and for the case that $D$ consists of the all-zero word. These sets are used to block-diagonalize the matrix $M_{D}(y)$ in either case. In this section we show that the sizes and the number of the blocks are bounded by a polynomial in $n_2$ and $n_3$. Furthermore, it is derived that the coefficients of the blocks can be computed efficiently. As before, we make a distinction between a code $D$ of size zero and one, starting with the latter.

\subsection{A code of size one}\label{subsection:sizeonecoef}

Let $D$ be the code consisting of the all-zero word $\underline{\boldsymbol{0}}$. Let $\Omega$ be the set of orbits of $\CC_3$ under the action of $G$. Recall that $S(D)$ consists of pairs of words containing $\underline{\boldsymbol{0}}$. For $w \in \Omega$, we define the $S(D) \times S(D)$ matrix $N_w$ by
\[
(N_{w})_{\{\underline{\boldsymbol{0}},x\},\{\underline{\boldsymbol{0}},y\}} \coloneqq \begin{cases} 1 \hspace{1mm}\text{ if } \{\underline{\boldsymbol{0}},x,y\} \in w\\
0 \hspace{1mm}\text{ otherwise}
\end{cases}
\]
\noindent Consider again the representative set from (\ref{equation:reprsetg}). Given $\mathbf{n} \in \mathbf{N_1}$ and $\pmb{\lambda} \vdash \mathbf{n}$, let $U_{\pmb{\lambda}}$ be the matrix corresponding to $\pmb{\lambda}$ and $\mathbf{n}$. Applying the map $\Phi$ from (\ref{equation:phireal}) to $M_{D}(y)$ gives\vspace{1mm}
\[
M_D(y) \mapsto \bigoplus_{\mathbf{n} \in \mathbf{N_{1}}}\bigoplus_{\pmb{\lambda} \vdash \mathbf{n}}U_{\pmb{\lambda}}^TM_D(y)U_{\pmb{\lambda}} = \bigoplus_{\mathbf{n} \in \mathbf{N_{1}}}\bigoplus_{\pmb{\lambda} \vdash \mathbf{n}}\sum_{w \in \Omega}y(w)U_{\pmb{\lambda}}^TN_wU_{\pmb{\lambda}}.
\]
This implies that we have to compute the blocks $U_{\pmb{\lambda}}^TN_wU_{\pmb{\lambda}}$ for all $\pmb{\lambda} \vdash \mathbf{n}$ and for all $w \in \Omega$.  We first argue that the sizes and number of these blocks are bounded by a polynomial in $n_2$ and $n_3$.\\
\indent From Section \ref{subsection:sizeone} it is clear that $|\mathbf{N_1}| = n_3+1$ and that for each $\mathbf{n} \in \mathbf{N_1}$, there is polynomial number (in $n_2$ and $n_3$) of $\pmb{\lambda}$ that partition $\mathbf{n}$. For each $\pmb{\lambda} \vdash \mathbf{n}$, with $\pmb{\lambda} = (\lambda_1,\lambda_2,\lambda_3)$ and such that the height of $\lambda_1$ and $\lambda_2$ is at most $2$ and the height of $\lambda_3$ is at most $1$, the cardinality of $W_{\pmb{\lambda}}'$ is seen to be bounded polynomially in $n_2$ and $n_3$ as well. Observe that $\Omega = \Omega_2 \times \Omega_3$, where $\Omega_i$ is the set of orbits of the collection of codes in $[i]^{n_i}$ of size at most $3$ under the action of $H_i = S_i^{n_i} \rtimes S_{n_i}$. The observations preceding Lemma $1$ of \cite{litjens17} show that $\Omega_2$ is polynomially bounded in size by $n_2$, and $\Omega_3$ similarly by $n_3$. This settles the first part of this section. Next we turn to computing the coefficients of the blocks $U_{\pmb{\lambda}}^TN_wU_{\pmb{\lambda}}$ for all $\mathbf{n} \in \mathbf{N_1}, \pmb{\lambda} \vdash \mathbf{n}$ and for all $w \in \Omega$.\\
\indent Given $\pmb{\lambda} \vdash \mathbf{n}$, calculating the coefficients amounts to computing the expressions $u_{\pmb{\sigma}}^TN_wu_{\pmb{\tau}}$, where $\pmb{\sigma}$ and $\pmb{\tau}$ range over $W_{\pmb{\lambda}}'$. We introduce some notation. Let $\Pi_2$ and $\Pi_3$ denote the collection of partitions of $\{1,2,3\}$ into at most $2$ parts and at most $3$ parts respectively. For $i = 2,3$ and for a word $v \in [i]^{3}$, let $\text{part}(v)$ denote the partition in $\Pi_i$ where $j$ and $l$ are in the same class of $\text{part}(v)$ if and only if $v_j = v_l$, for $1 \leq j,l \leq 3$. This gives a bijective correspondence between $\Pi_i$ and the number of orbits of $[i]^{3}$ under the natural action of $S_i$.\\
\indent For $P \in \Pi_2$, let $c_P$ be the average of $e_i \otimes e_j$ in $\R^{[2]} \otimes \R^{[2]}$ such that $\text{part}(0ij) = P$, with $i,j \in [2]$. Similarly, for $P \in \Pi_3$, let $d_P$ be the average of $f_i \otimes f_j$ in $\R^{[3]} \otimes \R^{[3]}$ such that $\text{part}(0ij) = P$, with $i,j \in [3]$.  Then the sets
\[
M_2 = \{c_P \hspace{1mm} | \hspace{1mm} P \in \Pi_2\} \text{ and } M_3 = \{d_P \hspace{1mm} | \hspace{1mm} P \in \Pi_3\}
\]
form orthogonal bases for $\R^{[2]} \otimes \R^{[2]}$ and $(\R^{[3]} \otimes \R^{[3]})^{S_2}$ respectively, where $S_2$ permutes the nonzero letters. Let $M_i^*$ denote the dual basis of $M_i$ for $i=2,3$. Let $Q_{2}$ denote the set of monomials of degree $n_2$ on $\R^{[2]} \otimes \R^{[2]}$ and $Q_3$ those of degree $n_3$ on $(\R^{[3]} \otimes \R^{[3]})^{S_2}$. Analogous to Section $4$ of \cite{litjens17}, the function $([2]^{n_2}[3]^{n_3})^{3} \rightarrow \CC_3$, that maps an \textit{ordered} triple $(\alpha,\beta,\gamma)$ to $\{\alpha,\beta,\gamma\}$, induces a surjective function 
\[
\kappa: Q_2 \times Q_3 \rightarrow \Omega \setminus \{\emptyset\}.
\]
For any $\mu \in Q_{2}$ and $\nu \in Q_{3}$, define
\[
K_{\mu,\nu} \coloneqq \sum_{\substack{c_{1},...,c_{n_{2}} \in M_2 \\
c_{1}^*\cdot...\cdot c_{n_{2}}^*=\mu}}\sum_{\substack{d_{1},...,d_{n_{3}} \in M_3 \\
d_{1}^*\cdot...\cdot d_{n_{3}}^*=\nu}} (\bigotimes_{j=1}^{n_{2}}c_{j}) \otimes (\bigotimes_{l=1}^{n_{3}}d_{l}).
\]

\begin{lemma}
Let $w \in \Omega$. Then we have that
\[
N_w = \sum_{\substack{(\mu,\nu) \in Q_{2} \times Q_{3}\\
\kappa(\mu,\nu) = w}}K_{\mu,\nu}.
\]
\end{lemma}
\begin{proof}
This follows directly from Lemma 1 of \cite{litjens17}.
\end{proof}

The lemma implies that it suffices to compute the expressions $u_{\pmb{\sigma}}^TK_{\mu,\nu}u_{\pmb{\tau}}$. Thereto, with respect to $\pmb{\sigma} = (\sigma_1,\sigma_2,\sigma_3)$ and $\pmb{\tau} = (\tau_1,\tau_2,\tau_3)$ and the matrices in (\ref{align:atjes}), we define the following polynomial
\[
p_{\pmb{\sigma},\pmb{\tau}} \coloneqq \prod_{j=1}^{3}\sum_{\substack{\sigma_j' \sim \sigma_j\\
\tau_j' \sim \tau_j}}\sum_{c_j, c_j' \in C_{\lambda_j}}\text{sgn}(c_jc_j')\prod_{y \in Y(\lambda_j)}A_j(\tau_j'c_j(y)) \otimes A_j(\sigma_j'c_j'(y)).
\]
Then $p_{\pmb{\sigma},\pmb{\tau}}$ is a polynomial of degree $n_2+n_3$ on $(\R^{[2]} \otimes \R^{[2]}) \otimes (\R^{[3]} \otimes \R^{[3]})^{S_2}$ and can be computed in terms of the $A_j(l) \otimes A_j(l)$ in polynomial (in $n_2$ and $n_3$) time (see Appendix $2$ of \cite{litjens17}). In view of Lemma $2$ of \cite{litjens17} we have
\[
\sum_{(\mu,\nu) \in Q_2 \times Q_3}(u_{\pmb{\sigma}}^TK_{\mu,\nu}u_{\pmb{\tau}})\mu\nu = p_{\pmb{\sigma},\pmb{\tau}}. 
\]
Hence we are faced with expressing the polynomials $p_{\pmb{\sigma},\pmb{\tau}}$ as linear combinations of the $\mu\nu \in Q_2Q_3$. In order to do so, we write the expressions $A_j(l) \otimes A_j(m)$ as linear functions in the bases $M_2^*$ and $M_3^*$, for all possible combinations of $j,l$ and $m$. The equations may be found in the appendix (Section \ref{section:appendix}).

\subsection{The empty code}

This section deals with the case that $D$ is the empty code. Since it is highly similar to the previous section, we omit some of the details. In the last part of this section it is explained how the empty code is added. For $w \in \Omega$, we define the $[2]^{n_2}[3]^{n_3} \times [2]^{n_2}[3]^{n_3}$ matrix $M_w$ by 
\[
(M_{w})_{x,y} \coloneqq \begin{cases} 1 \hspace{1mm}\text{ if } \{x,y\} \in w\\
0 \hspace{1mm}\text{ otherwise}
\end{cases}
\]
Consider again the representative set given in Proposition \ref{proposition:reprsetleeg}. Given $\mathbf{n} \in \mathbf{N_0}$ and $\pmb{\lambda} \vdash \mathbf{n}$, let $U_{\pmb{\lambda}}$ be the corresponding matrix. As before, the blocks $U_{\pmb{\lambda}}^TM_wU_{\pmb{\lambda}}$ are computed. Only the orbit corresponding to the empty set, the orbit corresponding to the singletons and the orbits of pairs of distinct words are taken into account.\\
\indent The number of orbits representing pairs of words equals the number of ordered partitions of the possible distances in at most two parts. This gives a number of orbits that is polynomial in $n_2$ and $n_3$. From Section \ref{subsection:sizezero} it is furthermore clear that $|\mathbf{N_{0}}| = (n_2+1)(n_3+1)$ and that for each $\mathbf{n} \in \mathbf{N_0}$, there is only one $\pmb{\lambda} = (\lambda_1,...,\lambda_4)$ that partitions $\mathbf{n}$ if all $\lambda_i$ are of height at most $1$. From this it follows that the cardinality of $Z_{\pmb{\lambda}}$ is one for any such $\pmb{\lambda}$, resolving the issue that only a polynomial number of blocks, that are of polynomial size, needs to be considered. We turn to the computation of the coefficients.\\
\indent With notation as in the previous section, let $\widetilde{\Pi} = \{\{123\},\{12,3\}\} \subset \Pi_2$. The sets 
\[
\widetilde{M_2} = \{c_{P} \mid P \in \widetilde{\Pi}\} \text{ and } \widetilde{M_3} = \{d_P \mid P \in \widetilde{\Pi}\}
\] 
form orthogonal bases for $(\R^{[2]} \otimes \R^{[2]})^{S_2}$ and $(\R^{[3]} \otimes \R^{[3]})^{S_3}$ respectively. Let $\widetilde{M_i}^*$ denote the dual basis of $\widetilde{M_i}$, for $i=2,3$. Similar to the previous section, we are ultimately led to the problem of expressing the tensors $B_j(1) \otimes B_j(1)$ (see (\ref{align:b1234})) as linear functions in the bases $\widetilde{M_2}^*$ and $\widetilde{M_3}^*$, for $1 \leq j \leq 4$. The equations are found in the appendix (Section \ref{section:appendix}).\\
\indent Lastly, the empty code is added. As mentioned at the end of Section \ref{subsection:sizezero}, we create an extra row and column corresponding to the vector $e_{\emptyset}$ to the matrix indexed by the partition $\pmb{\lambda} = ((n_2),(),(n_3),())$. The upper left coefficient is equal to $e_{\emptyset}^TM_D(y)e_{\emptyset} = y(\emptyset) = 1$, by (i) of (\ref{align:upperbound}). For $\pmb{\lambda} = ((n_2),(),(n_3),())$, the cardinality of $Z_{\pmb{\lambda}}$ is one, hence there is only one more coefficient to compute. Let $\pmb{\sigma}$ be the unique element in $Z_{\pmb{\lambda}}$, then $v_{\pmb{\sigma}} = \sum_{u \in [2]^{n_2}[3]^{n_3}}e_{u}$ and we compute
\[
e_{\emptyset}^TM_D(y)v_{\pmb{\sigma}} = \sum_{u \in [2]^{n_2}[3]^{n_3}}y(\{u\}) = 2^{n_2}3^{n_3}y(w),
\]
where $w$ is the orbit corresponding to singletons of words.

\section{Table}

The following table shows the improvements that were found on the known upper bounds of $N(n_2,n_3,d)$. In total, $135$ new bounds were obtained. The $131$ unmarked bounds are directly from the semidefinite program. The bound on $(n_2,n_3,d) = (4,3,3)$ is marked with $_1$ and was found using the optimization problem (\ref{align:upperbound}) for quadruples of words ($k=4$). Although the computations for this case are not included in the article, we included the result in the table. The bound on  $(n_2,n_3,d) = (2,12,8)$ is marked with $_2$ and follows from the general inequality $N(n_2+1,n_3,d) \leq 2N(n_2,n_3,d)$ together with $N(1,12,8) \leq 67$. A. Brouwer observed that actually two more new upper bounds follow from this inequality. Namely, $N(5,3,3) \leq 2N(4,3,3) \leq 60$ and $N(5,9,4) \leq 2N(4,9,4) \leq 9180$. These bounds are marked with $_2$ as well.
\vspace{-3mm}
\begin{center}
\small
\captionof{table}{New upper bounds on $N(n_2,n_3,d)$}\label{table:bounds}
\begin{tabular}[!htb]{ >{\raggedleft\arraybackslash}p{0.2cm} >{\raggedleft\arraybackslash}p{0.2cm}  >{\raggedleft\arraybackslash}p{0.2cm}  >{\raggedleft\arraybackslash}p{1.1cm}  >{\raggedleft\arraybackslash}p{1.2cm} >{\raggedleft\arraybackslash}p{1.8cm}}
\hline
            \rule{0pt}{3ex}$n_2$ & $n_3$ & $d$ & \multicolumn{1}{>{\centering\arraybackslash}p{1.1cm}}{Best lower bound known} &  \multicolumn{1}{>{\centering\arraybackslash}p{1.2cm}}{\textbf{New upper bound}} & \multicolumn{1}{>{\centering\arraybackslash}p{1.8cm}}{Best upper bound previously known} \\ \hline 
\rule{0pt}{3ex}2 & 5 & 3 & 52 & 65 & 66\\
3 & 5 & 3 & 99 & 125 & 126\\
4 & 3 & 3 & 28 & $30_{1}$ & 33\\
4 & 5 & 3 & 186 & 238 & 243\\
4 & 8 & 3 & 3888 & 4764 & 4767\\
5 & 3 & 3 & 54 & $60_2$ & 65\\
5 & 4 & 3 & 144 & 165 &167\\
6 & 3 & 3 & 108 & 118 & 123\\
6 & 4 & 3 & 288 & 317 & 322\\
6 & 5 & 3 & 672 & 855 & 863\\
7 & 2 & 3 & 72 & 83 & 85\\
7 & 3 & 3 & 192 & 225 & 230\\
7 & 4 & 3 & 576 & 604 & 609\\
8 & 1 & 3 & 50 & 59 & 60\\
8 & 2 & 3 & 144 & 154 & 160\\
8 & 3 & 3 & 384 & 414 & 417\\
8 & 5 & 3 & 2560 & 3087 & 3110\\
9 & 1 & 3 & 96 & 108 & 109\\
9 & 2 & 3 & 288 & 292 & 293\\
9 & 3 & 3 & 768 & 796 & 806\\
9 & 4 & 3 & 1728 & 2130 & 2131\\ 
10 & 1 & 3 & 192 & 212 & 213\\
10 & 2 & 3 & 512 & 552 & 556\\
10 & 3 & 3 & 1152 & 1492 & 1536\\
10 & 4 & 3 & 3280 & 4081 & 4147\\
11 & 3 & 3 & 2304 & 2890 & 2910\\
\rule{0pt}{0ex}13 & 1 & 3 & 1120 & 1360 & 1365\\
\hline
\rule{0pt}{2ex}1 & 12 & 4 & 8019 & 13531 & 13678\\
1 & 13 & 4 & 16767 & 37714 & 38540\\
2 & 6 & 4 & 51 & 61 & 66\\
2 & 10 & 4 & 1944 & 3371 & 3498\\
2 & 11 & 4 & 5589 & 9450 & 9777\\
3 & 5 & 4 & 36 & 43 & 44\\
3 & 6 & 4 & 92 & 117 & 124\\
3 & 10 & 4 & 3726 & 6581 & 6791\\
3 & 11 & 4 & 10692 & 18039 & 19554\\
4 & 5 & 4 & 62 & 83 & 86\\
4 & 6 & 4 & 158 & 228 & 242\\
4 & 9 & 4 & 2484 & 4590 & 4752\\
\end{tabular}
\hspace{5mm}
\begin{tabular}[!htb]{ >{\raggedleft\arraybackslash}p{0.2cm}  >{\raggedleft\arraybackslash}p{0.2cm}  >{\raggedleft\arraybackslash}p{0.2cm} >{\raggedleft\arraybackslash}p{1.1cm}  >{\raggedleft\arraybackslash}p{1.2cm} >{\raggedleft\arraybackslash}p{1.8cm}}
\hline
            \rule{0pt}{3ex}$n_2$ & $n_3$ & $d$ & \multicolumn{1}{>{\centering\arraybackslash}p{1.1cm}}{Best lower bound known} &  \multicolumn{1}{>{\centering\arraybackslash}p{1.2cm}}{\textbf{New upper bound}} & \multicolumn{1}{>{\centering\arraybackslash}p{1.8cm}}{Best upper bound previously known} \\ \hline
\rule{0pt}{3ex}5 & 4 & 4 & 50 & 59 & 60\\
5 & 5 & 4 & 114 & 160 & 167\\
5 & 6 & 4 & 288 & 436 & 454\\
6 & 4 & 4 & 96 & 114 & 120\\
6 & 5 & 4 & 216 & 308 & 319\\
6 & 6 & 4 & 576 & 825 & 863\\
7 & 4 & 4 & 192 & 220 & 230\\
7 & 5 & 4 & 408 & 585 & 612\\
7 & 6 & 4 & 1152 & 1576 & 1612\\
8 & 2 & 4 & 50 & 59 & 60\\
8 & 3 & 4 & 128 & 153 & 160\\
8 & 4 & 4 & 384 & 407 & 417\\
8 & 5 & 4 & 768 & 1103 & 1120\\
8 & 6 & 4 & 2304 & 3027 & 3224\\
9 & 2 & 4 &  96 & 108 & 109\\
9 & 3 & 4 & 256 & 288 & 293\\
9 & 4 & 4 & 548 & 771 & 782\\
9 & 5 & 4 & 1536 & 2105 & 2199\\
10 & 2 & 4 & 192 & 212 & 213\\
10 & 3 & 4 & 420 & 548 & 556\\
10 & 4 & 4 & 1050 & 1480 & 1533\\
\rule{0pt}{0ex}11 & 3 & 4 & 784 & 1032 & 1060\\
\hline
\rule{0pt}{2ex}1 & 11 & 5 & 729 & 1138 & 1145\\
1 & 12 & 5 & 1458 & 2927 & 2984\\
1 & 13 & 5 & 4374 & 7598 & 7630\\
2 & 10 & 5 & 729 & 849 & 867\\
2 & 11 & 5 & 972 & 2105 & 2157\\
2 & 12 & 5 & 2916 & 5512 & 5636\\
3 & 9 & 5 & 486 & 601 & 633\\
3 & 10 & 5 & 729 & 1519 & 1567\\
3 & 11 & 5 & 1944 & 3964 & 4122\\
4 & 8 & 5 & 324 & 420 & 432\\
4 & 9 & 5 & 729 & 1099 & 1153\\
4 & 10 & 5 & 1458 & 2801 & 2921\\
5 & 8 & 5 & 486 & 791 & 850\\
5 & 9 & 5 & 1458 & 2000 & 2098\\
6 & 7 & 5 & 378 & 563 & 576\\
6 & 8 & 5 & 972 & 1437 & 1481\\
7 & 6 & 5 & 255 & 407 & 432\\
\end{tabular}
\end{center} 

\begin{center}
\small
\begin{minipage}[t]{0.45\linewidth}
\begin{tabular}[t]{ >{\raggedleft\arraybackslash}p{0.2cm}  >{\raggedleft\arraybackslash}p{0.2cm}  >{\raggedleft\arraybackslash}p{0.2cm}  >{\raggedleft\arraybackslash}p{1.1cm}  >{\raggedleft\arraybackslash}p{1.2cm} >{\raggedleft\arraybackslash}p{1.8cm}}

\hline
            \rule{0pt}{3ex}$n_2$ & $n_3$ & $d$ & \multicolumn{1}{>{\centering\arraybackslash}p{1.1cm}}{Best lower bound known} &  \multicolumn{1}{>{\centering\arraybackslash}p{1.2cm}}{\textbf{New upper bound}} & \multicolumn{1}{>{\centering\arraybackslash}p{1.8cm}}{Best upper bound previously known} \\ \hline 
\rule{0pt}{3ex}7 & 7 & 5 & 648 & 1047 & 1089\\
8 & 3 & 5 & 34 & 44 & 48\\
8 & 6 & 5 & 453 & 755 & 806\\
9 & 2 & 5 & 26 & 31 & 32\\
9 & 3 & 5 & 64 & 85 & 91\\
9 & 4 & 5 & 136 & 216 & 224\\
9 & 5 & 5 & 318 & 534 & 576\\
10 & 2 & 5 & 48 & 61 & 64\\
10 & 3 & 5 & 128 & 158 & 170\\
10 & 4 & 5 & 234 & 390 & 427\\
11 & 1 & 5 & 38 & 43 & 48\\
11 & 2 & 5 & 96 & 115 & 121\\
11 & 3 & 5 & 192 & 292 & 316\\
12 & 1 & 5 & 64 & 83 & 86\\
12 & 2 & 5 & 192 & 213 & 236\\
\rule{0pt}{0ex}13 & 1 & 5 & 128 & 156 & 170\\ 
\hline
\rule{0pt}{2ex}1 & 12 & 6 & 729 & 1073 & 1145\\
1 & 13 & 6 & 1458 & 2657 & 2868\\
2 & 11 & 6 & 729 & 803 & 867\\
2 & 12 & 6 & 972 & 1935 & 2093\\
3 & 10 & 6 & 486 & 574 & 614\\
3 & 11 & 6 & 729 & 1414 & 1512\\
4 & 10 & 6 & 729 & 1036 & 1133\\
5 & 8 & 6 & 216 & 276 & 288\\
5 & 9 & 6 & 486 & 744 & 829\\
6 & 8 & 6 & 324 & 527 & 576\\
7 & 4 & 6 & 18 & 22 & 24\\
7 & 6 & 6 & 99 & 142 & 144\\
\end{tabular}
\end{minipage}
\hspace{5mm}
\begin{minipage}[t]{0.45\linewidth}
\begin{tabular}[t]{ >{\raggedleft\arraybackslash}p{0.2cm}  >{\raggedleft\arraybackslash}p{0.2cm}  >{\raggedleft\arraybackslash}p{0.2cm}  >{\raggedleft\arraybackslash}p{1.1cm}  >{\raggedleft\arraybackslash}p{1.2cm} >{\raggedleft\arraybackslash}p{1.8cm}}
\hline
            \rule{0pt}{3ex}$n_2$ & $n_3$ & $d$ & \multicolumn{1}{>{\centering\arraybackslash}p{1.1cm}}{Best lower bound known} &  \multicolumn{1}{>{\centering\arraybackslash}p{1.2cm}}{\textbf{New upper bound}} & \multicolumn{1}{>{\centering\arraybackslash}p{1.8cm}}{Best upper bound previously known} \\ \hline 
\rule{0pt}{3ex}7 & 7 & 6 & 216 & 375 & 384\\
8 & 4 & 6 & 32 & 39 & 43\\
8 & 6 & 6 & 168 & 273 & 288\\
9 & 3 & 6 & 26 & 30 & 32\\
9 & 4 & 6 & 56 & 75 & 77\\
10 & 3 & 6 & 44 & 56 & 61\\
10 & 4 & 6 & 88 & 144 & 153\\
11 & 2 & 6 & 32 & 43 & 48\\
11 & 3 & 6 & 88 & 107 & 112\\
\rule{0pt}{0ex}12 & 2 & 6 & 64 & 83 & 87\\ 
\hline
\rule{0pt}{2ex}1 & 13 & 7 & 243 & 591 & 623\\
5 & 9 & 7 & 69 & 174 & 180\\
6 & 6 & 7 & 18 & 23 & 24\\
6 & 7 & 7 & 33 & 53 & 56\\
6 & 8 & 7 & 61 & 130 & 135\\
7 & 6 & 7 & 24 & 41 & 45\\
7 & 7 & 7 & 58 & 99 & 102\\
8 & 5 & 7 & 22 & 31 & 32\\
8 & 6 & 7 & 44 & 74 & 79\\
9 & 4 & 7 & 18 & 23 & 26\\
9 & 5 & 7 & 36 & 53 & 62\\
10 & 4 & 7 & 28 & 41 & 47\\
11 & 3 & 7 & 24 & 31 & 35\\
\rule{0pt}{0ex}13 & 1 & 7 & 16 & 19 & 20\\ 
\hline
\rule{0pt}{2ex}1 & 12 & 8 & 39 & 67 & 72\\
2 & 12 & 8 & 36 & 134$_{2}$ & 139\\
\rule{0pt}{0ex}6 & 8 & 8 & 28 & 44 & 46\\ 
\hline
\rule{0pt}{2ex}1 & 13 & 9 & 30 & 50 & 54
\end{tabular}
\end{minipage}
\end{center}

\section{Appendix}\label{section:appendix}

In this appendix we express all $A_j(l) \otimes A_j(m)$ and $B_k(1) \otimes B_k(1)$ as linear functions in the bases $M_2^*, M_3^*$ and $\widetilde{M_2}^*, \widetilde{M_3}^*$ respectively. This is done by evaluating the tensors at the basis elements of $M_2, M_3$ and $\widetilde{M_2},\widetilde{M_3}$. A partition is denoted by a sequence of its classes. For example, $c_{12,3}^*$ stands for the dual variable corresponding to the partition $\{\{1,2\},\{3\}\}$ of $\{1,2,3\}$. It is found that
\begin{small}
\begin{align*}
&A_1(1) \otimes A_1(1) = c_{123}^* && A_3(1) \otimes A_3(1) = 2(d_{1,23}^* - d_{1,2,3}^*)\\
&A_1(1) \otimes A_1(2) = c_{12,3}^* && B_1(1) \otimes B_1(1) =  2(c_{123}^*+ c_{12,3}^*)\\
&A_1(2) \otimes A_1(1) = c_{13,2}^* && B_2(1) \otimes B_2(1) = 2(c_{123}^*- c_{12,3}^*)\\
&A_1(2) \otimes A_1(2) = c_{1,23}^* && B_3(1) \otimes B_3(1) = 3(d_{123}^* + 2d_{12,3}^*)\\
&A_2(1) \otimes A_2(1) = d_{123}^* && B_4(1) \otimes B_4(1) = 2(d_{123}^* - d_{12,3}^*)\\
&A_2(1) \otimes A_2(2) = 2d_{12,3}^*\\
&A_2(2) \otimes A_2(1) = 2d_{13,2}^*\\
&A_2(2) \otimes A_2(2) = 2(d_{1,23}^*+d_{1,2,3}^*)
\end{align*}
\end{small}
\noindent \hspace{-2mm} \textit{Acknowledgements.} The author would like to thank Lex Schrijver and Sven Polak for useful discussions and SURFsara for the support in using the LISA Compute Cluster. Furthermore, we thank Andries Brouwer for notifying the author about two improved upper bounds that follow from our calculations.

\end{document}